\documentclass[a4paper,american,reqno]{amsart}
\usepackage{float}
\usepackage{caption}
\usepackage{pgfplots}
\newif\ifPreprint \Preprintfalse
\newif\ifSubmission \Submissiontrue

\usepackage{babel}
\usepackage[utf8]{inputenc}
\usepackage[T1]{fontenc}
\usepackage{lmodern}
\usepackage{xspace}
\usepackage{mathtools}
\usepackage{amsmath}
\usepackage{enumitem}
\usepackage{accents}
\usepackage{todonotes}
\usepackage{csquotes}
\usepackage{microtype}
\usepackage[colorlinks,
            citecolor=blue,
            urlcolor=blue,
            linkcolor=blue]{hyperref}
\usepackage{cleveref}
\usepackage{bbm}

\makeatletter
\patchcmd{\@settitle}{\uppercasenonmath\@title}{\scshape\large}{}{}
\patchcmd{\@setauthors}{\MakeUppercase}{\scshape\normalsize}{}{}
\makeatother

\vfuzz \hfuzz
\theoremstyle{plain}
\newtheorem{lemma}{Lemma}

\newtheorem{theorem}{Theorem}

\theoremstyle{definition}

\newtheorem{example}{Example}

\newtheorem{prop}{Proposition}
\newtheorem{remark}{Remark}

\theoremstyle{remark}

\newcommand{\st}{\text{s.t.}}
\newcommand{\Tr}{\text{Tr}}

\newcommand{\R}{\mathbb{R}}

\newcommand{\N}{\mathbb{N}}
\newcommand{\Z}{\mathbb{Z}}

\newcommand{\Ccal}{\mathcal{C}}
\newcommand{\Mcal}{\mathcal{M}}
\newcommand{\Ncal}{\mathcal{N}}
\newcommand{\Pcal}{\mathcal{P}}
\newcommand{\Scal}{\mathcal{S}}

\usepackage[backend=biber]{biblatex}
\begin{document}

\title[Safe Approximation of Multivariate DRO]{A Positive Semidefinite Safe Approximation of
  Multivariate Distributionally Robust Constraints Determined by Simple Functions} 
\author[J. Dienstbier, F. Liers, J. Rolfes]%
{J. Dienstbier, F. Liers, J. Rolfes}

\address[J. Dienstbier, F. Liers, J. Rolfes]{%
  Friedrich-Alexander-Universität Erlangen-Nürnberg,
  Cauerstr. 11,
  91058 Erlangen,
  Germany}
\email{\{jana.jd.dienstbier, frauke.liers, jan.rolfes\}@fau.de}

\newcommand{\JD}[1]{\todo[author=JD,color=green!50,size=\small]{#1}}
\newcommand{\JDil}[1]{\todo[inline,author=JD,color=green!50,size=\small]{#1}}
\newcommand{\FL}[1]{\todo[author=FL,color=red!50,size=\small]{#1}}
\newcommand{\FLil}[1]{\todo[inline,author=FL,color=red!50,size=\small]{#1}}
\newcommand{\JR}[1]{\todo[author=JR,color=orange!50,size=\small]{#1}}
\newcommand{\JRil}[1]{\todo[inline,author=JR,color=orange!50,size=\small]{#1}}
\newcommand{\RTD}[1]{\rho({#1})}
\newcommand{\uRTD}[2]{\rho_{#1}({#2})}
\newcommand{\mass}[1]{\mathbbm{P}_{#1}}
\newcommand{\EW}[1]{\mu_{#1}}
\newcommand{\varianz}[1]{\sigma_{#1}}
\newcommand{\bound}[2]{\bar{\rho}_{#1}({#2})}
\newcommand{\boundvariable}[1]{z_{#1}}
\newcommand{\lift}[2]{\Delta_{#1}^{#2}}
\newcommand{\minparabel}{p_{\text{min}}}
\newcommand{\sign}{\text{sign}}
\newcommand{\E}{\mathbbm{E}}
\renewcommand{\P}{\mathbbm{P}}

\date{\today}

\begin{abstract}
  Single-level reformulations of (nonconvex) distributionally robust optimization (DRO)
problems are often intractable, as they contain semiinfinite dual constraints. Based on such a semiinfinite reformulation, we present a safe approximation, that allows for the computation of feasible solutions for DROs that depend on nonconvex multivariate simple functions. Moreover, the approximation allows to address ambiguity sets that can incorporate information on moments as well as confidence sets. 
The typical strong assumptions on the structure of the underlying
constraints, such as convexity in the decisions or concavity in the
uncertainty found in the literature were, at least in part, recently
overcome in \cite{Dienstbier2023a}. We start from the duality-based
reformulation approach
in \cite{Dienstbier2023a} that can be applied for DRO constraints based on simple
functions that are univariate in the uncertainty parameters. We
significantly extend their approach to multivariate simple functions which
leads to a considerably wider applicability of the proposed
reformulation approach. In
order to achieve algorithmic tractability, the presented safe
approximation is then realized by a discretized counterpart for the
semiinfinite dual constraints. The approximation leads to a
computationally tractable mixed-integer positive semidefinite problem
for which state-of-the-art software implementations are readily available.
The tractable safe approximation provides sufficient conditions for
distributional robustness of the original problem, i.e., obtained
solutions are provably robust.

\end{abstract}


\makeatletter
\@namedef{subjclassname@2020}{%
	\textup{2020} Mathematics Subject Classification}
\makeatother


\maketitle


\section{Introduction}
\label{Sec:introduction}

In this work, we consider distributionally robust
optimization (DRO) models that are governed by multivariate simple
functions that appear in many relevant contexts. Despite their nonconvexity, we aim for algorithmically
tractable approximations that are based on duality arguments. The resulting
solutions yield a safe approximation which means that they are
guaranteed to be robust for the original constraints.

The approach presented here starts from the considerations in
\cite{Dienstbier2023a} for constraints that are univariate in the
uncertain parameters and generalizes the approach to the considerably
more general case of constraints that are multivariate in the
uncertainty. In the latter approach, a safe approximation was
developed that leads to a mixed-integer linear optimization
problem. Despite the NP-hardness of the latter, 
practically efficient algorithms and software are readily
available. In addition, it could be proven that the safe approximation
is asymptotically correct, i.e., it does not only yield robust
solutions, but aysmptotically solves the original distributionally
robust problem. In our generalization to the multivariate setting, we
use the same notation as in \cite{Dienstbier2023a}. For completeness
of the exposition, we repeat the necessary ingredients. Let $x\in
\R^n$ denote the decision variables, $b\in \R$ a scalar, $\Pcal$ a set of probability measures on the compact domain $T\subseteq \R^m$. {We then model the uncertainty in our DRO with a random vector $t\in T$ distributed according to an (uncertain) probability measure $\P\in \Pcal$.} 
As typical in (distributionally) robust optimization, the
task consists in determining decisions $x$ that are feasible even in case the uncertain {probability measures} are chosen in an adversarial way which coined the name 'adversary'. In addition, in case of an
optimization model, the chosen robust solution shall lead to a best possible guaranteed objective value. 
Here, $v:\R^n\times T\rightarrow \R$ denotes a (possibly nonconvex) function that connects the decision variables $x$ with the random vector $t$. Then, a \emph{distributionally robust constraint} or DRO constraint is defined by

\begin{equation}\label{Eq: DRO_constraint}
		b\leq \min_{\P\in \Pcal} \E_{\P}\left(v(x,t)\right).
\end{equation}

Constraints of this form contain both the purely stochastic as well as
the robust models as special cases. Indeed,
setting $\Pcal=\{\P\}$, leads to a \emph{stochastic
  constraint}: 
\begin{equation*}
	b\leq \E_{\P}\left(v(x,t)\right).
\end{equation*}
Stochastic optimization has been established
for situations when uncertainty
is distributed by some (known) distribution or when constraints must be met
with a certain probability. It hedges against uncertainty in a
probabilistic sense and implicitely assumes that the underlying distributions can be closely
approximated or is even known exactly. We refer to
\cite{birge2006introduction} for a gentle introduction on stochastic
optimization and to the surveys \cite{prekopa1998SO} and
\cite{Shapiro2003} particularly for discrete random variables.

Setting $\Pcal=\{\delta_t: t\in T\}$, where $\delta_t$ denotes the
Dirac point measures at $t\in T$, \eqref{Eq: DRO_constraint} yields a \emph{robust constraint}
\begin{equation*}
	b\leq \min_{t\in T} v(x,t).
\end{equation*}
\bigskip
While more details on related literature will be given in a later
section, we mention here the
introductory textbooks for
continuous robust optimization~\cite{ben2000robust},
\cite{bertsimas2022robust} as well as on combinatorial robust
optimization~\cite{kouvelis97}, \cite{GoerigkHartisch24}.

In stochastic optimization, a large variety of efficient and elegant
models and solution approaches have been established. However, in
applications the underlying distribution are often unknown, which may
result in low-quality or even infeasible results in case the
underlying assumptions on the
distributions are not satisfied.

In contrast, robust
optimization offers a natural alternative, whereby uncertainty sets
are established a priori. Feasibility of an obtained solution is guaranteed for all possible outcomes of
uncertainty within the uncertainty sets. A solution with best
guaranteed value is determined. Modelling and algorithmical aproaches
consist in
(duality-based) reformulations of the semi-infinite or exponentially
large robust counterparts, if it is allowed by underlying structural
assumptions such as convexity or more generally some underlying
duality theory. 
If this is not possible, then decomposition algorithms
are developed, possibly together with some approximation
approaches if still the underlying robust problems are too demanding
to solve.  

In this work, we focus on distributional robust optimization (DRO). DRO determines robust solutions that are protected against uncertainty in the underlying distributions. These distributions are
assumed to reside in a so-called ambiguity
set of probability measures, denoted by $\Pcal$ above.
For distributionally robust optimization, we refer to the detailed surveys \cite{Rahimian2019a} and
\cite{Fengmin2021a} as well as the references therein.

Generalizing from \cite{Dienstbier2023a}, we here allow the presence of multivariate simple
functions $v$, i.e. $\text{dim}(T)=m>1$. The latter are basic building blocks in Lebesgue integrals.
As simple functions are nonconvex, we cannot expect to derive an equivalent reformulation of the DRO model. However, our main contribution lies in the derivation of a mixed-integer positive semidefinite
safe approximation, i.e., all obtained solutions are guaranteed to be robust.
Due to the availability of state-of-the-art
software implementations for mixed-integer positive semidefinite
optimization, this proves the computational tractability of our
modelling approach. 


This work is structured as follows.
Section \ref{Sec: Problem Setting} introduces the distributionally
robust model including simple functions, together with motivation
and illustrative examples. Subsequently, Section \ref{Sec: DRO_indicator_functions} 
presents a new semi-infinite inner approximation of the robust
counterpart, along with a suitable discretization. The result consists
in a novel finite-dimensional mixed-integer positive semidefinite
optimization model. The main contribution consists in showing that its
feasible solutions are also 
feasible for the original robust DRO model.

\section{Literature Review}
\label{Sec:literaturereview}

Next,
we briefly review some relevant literature in optimization under
uncertainty, and distributional robustness in particular. 
Next to the textbooks \cite{ben2000robust},
\cite{bertsimas2022robust}, \cite{kouvelis97}, \cite{GoerigkHartisch24}
mentioned above, relevant
literature on robust optimization starts from the first treatments of linear
optimization with uncertain coefficients~\cite{soyster1973convex} to a
systematic study of linear optimization under uncertainty in e.g., 
\cite{ben1998robust}, \cite{ben1999robust}. In these approaches,
duality-based reformulations have been
developed that lead to algorithmically tractable robust counterparts
for wich practically usable solution approaches exist or
software packages are available.

In order to push duality-based reformulation approaches even beyond linear
and convex optimization, a wide variety of reformulations have been presented in
\cite{Fenchel}. If an underlying duality theory
cannot be assumed, often decomposition approaches are developed. This
is in particular the case for nonconvex robust optimization where the
optimization problem is nonconvex in the uncertainty. A practically
efficient solution framework is given by an adaptive bundle approach~\cite{kuchlbauer2022adaptive}
which has been integrated in an outer approximation procedure in \cite{kuchlbauerOuter} for
additional discrete decisions. We refer to the survey
\cite{NonlinearRO} for additional references for nonlinear robustness.


Robust and stochastic constraints can be integrated either via so-called 'probust
functions', see e.g. \cite{adelhuette2021joint} or \cite{Berthold2021}, or via distributional
robustness (DRO) from formula \eqref{Eq: DRO_constraint}. Such integrated robust-probabilistic
models contain advantages of both worlds, namely full protection as in
the robust world together with limited prize of uncertainty protection
as in stochastic optimization.


DRO surveys are presented in \cite{Rahimian2019a} and
\cite{Fengmin2021a}.


It is widely accepted that the right choice of ambiguity set is
crucial both with respect to algorithmic tractability of the resulting
robust counterparts as well as with
respect to the obtained solution quality. Indeed, the ambiguity sets
shall be chosen that the most relevant uncertainties are considered,
while taking available partial information into account, but
simultaneously that overconservative solutions are avoided. 

Discrepancy-based ambiguity sets assume a nominal, 'typical', distribution
and include distributions within a certain distance of
it, where Wasserstein-balls are natural distances
\cite{mohajerin2018data}.
Ambiguity sets have for example also been derived from phi-divergence \cite{gotoh2018},
from likelihood ambiguity sets in \cite{Wang2015}, as well as from statistical hypothesis tests \cite{Bertsimas2017}.

Going beyond convex models, in our approach we allow the presence of nonconvex simple functions and mainly focus on moment-based ambiguity sets. The
moments of distributions are uncertain but are assumed to satisfy predetermined
bounds. For convex models, mean-variance or Value-at-Risk measures are
studied in \cite{ghaoui2003worst}, whereas moment information is used in
\cite{Popescu2007a}. The article \cite{Xu2017a} uses Slater conditions to show the
correctness of a duality-based reformulation of the robust
counterpart, together with discretization schemes to determine
approximate solutions. \cite{Delage2010a} presents exact reformulations of
convex DRO problems, where confidence regions of some moments are considered.



For convex models, some recent works combine partial information based on discrepancies as
well as on moments of the distribution, to define 'tight' ambiguity sets.
In this flavor, in \cite{Cheramin2022} the authors derive efficient inner and outer approximations
for DRO where both moment as well as Wasserstein ambiguity sets can be
used simultaneously. 

One of the challenges of incorporating additional information into
moment-based ambiguity sets is addressed by the authors of
\cite{Parys2015a}, who provide a positive semidefinite (SDP) reformulation of \eqref{Eq:
  DRO_constraint} for cases where the probability distribution is
known to be unimodal and the moments are fixed. 
On the other hand, \cite{Wiesemann2014a} presents a
duality-based reformulation of \eqref{Eq: DRO_constraint} that
incorporates information on the confidence sets and assumes
convex optimization problems. Under these assumptions, the approach can be applied to a DRO with \eqref{Eq: DRO_constraint} as a constraint.

The recent work \cite{Luo2023} also allows hybrid ambiguity
sets by enriching Wasserstein balls with additional moment information. For discrete
decisions, approximations are presented. 

While many existing approaches consider static DRO problems,
\cite{aigner2023} learn ambiguity sets and robust decisions for DRO problems with discrete probability distributions that repeat over time. 

In \cite{Bayraksan}, the authors go a step further and consider
multi-stage DRO settings. Via scenario grouping, they present bounds taking conditional
ambiguity sets into account that occur in the multi-stage
mixed-integer DRO setting. 

In our work, we use moment-based ambiguity sets similar to 
\cite{Delage2010a}, which consider mean and covariance matrix ranges
along with confidence set information as in
\cite{Wiesemann2014a}. 
Our main contribution is to allow the presence of nonconvexities which considerably extends
existing reformulation approaches. Indeed, in addition to
being able to model tight ambiguity sets, we also allow that the optimization
models contain
multivariate nonconvex simple functions. These functions can approximate any, even
nonconvex, continuous function. 

Due to these
nonconvexities, standard reformulation approaches based on duality
cannot be applied. In order to apply them nevertheless, we first approximate the nonconvexities
appropriately by convex functions for which we then present reformulations to
optimization problems in function space. \cite{Dienstbier2023a} considers the
univariate case and presents a safe approximation that is based on
mixed-integer linear constraints. In addition, they could prove that
the safe approximation converges to the true robust counterpart
solution, rendering the approximation asymptotically a correct
equivalent reformulation. For the multivariate setting considered
here, appropriate discretizations result in 
mixed-integer positive-semidefinite optimization problems. The latter are algorithmically
tractable and can be solved via available software. As a result, we
present reformulation approaches for such nonconvex multivariate
DRO problems that allow algorithmically tractable
solution of the resulting robust counterparts. 


\section{Problem Setting and Notation}\label{Sec: Problem Setting}
We stick to the notation from \cite{Dienstbier2023a} and
summarize the main modelling here for completeness of our exposition.

\subsection{DRO Constraints Containing Simple Functions}
The DRO constraints considered in the present article are defined by functions $v(x,t)$ that consist of multivariate \emph{simple functions}, i.e., finite linear combinations of indicator functions:
$$v(x,t)=\sum_{i=1}^k x_i \mathbbm{1}_{X_i}(t), \text{ where } \mathbbm{1}_{X_i}(t)\coloneqq \begin{cases}
	1 & \text{ if } t\in X_i\\
	0 & \text{ otherwise.}
\end{cases}$$
The functions of type $\mathbbm{1}_{X_i}$ are denoted as \emph{indicator functions} as they indicate whether $t\in X_i$ holds or not. The sets $X_i$ can be considered as events in the probability space given by $\P$. In fact, considering functions $v$ as above in \eqref{Eq: DRO_constraint} leads to 
$$\E_{\P}(v(x,t)) = \E_{\P}\left(\sum_{i=1}^k x_i\mathbbm{1}_{X_i}(t)\right) = \sum_{i=1}^k x_i \mathbbm{P}(X_i)$$
and consequently the following formulation of \eqref{Eq: DRO_constraint}:
\begin{equation}\label{Eq: DRO_constraint_Prob}
	b \leq \min_{\mathbbm{P}\in \Pcal} \sum_{i=1}^k x_i \mathbbm{P}(X_i).
\end{equation} 
We note, that one may see \eqref{Eq: DRO_constraint_Prob} as a robust chance constraint, that is allowed to consist of simple functions. The decisions may either influence the height $x_i$ of an indicator
function, see Case 1, or will determine the underlying domains
$X_i$, see Case 2. In the remainder of this paper, we will investigate both
situations separately to ease the presentation. However, the safe approximation presented in Theorem \ref{Thm: MIP_multidim} can be extended to incorporate both cases simultaneously.

\textbf{Case 1:} Suppose that the sets $X_i\subseteq \R^m$ are given, then we ask for optimal decisions $x_i$ such that the DRO constraint \eqref{Eq: DRO_constraint} is satisfied.
\begin{subequations}\label{Prob: Case1}
	\begin{align}
		\max_{x\in C}\ & c(x): \\
		\st\ & b \leq \min_{\mathbbm{P}\in \Pcal} \sum_{i=1}^k x_i \mathbbm{P}(X_i),\label{Constr: Case1}
	\end{align}
\end{subequations}
where, $c:\R^n\rightarrow \R$ denotes a concave objective function, $C\subseteq \R^n$ denotes a set of additional convex constraints. Note, that in Case 1, we have that $n=k$.

We demonstrate the generality of \eqref{Prob: Case1} by an academic example on the mean-variance model from portfolio optimization, see Example 3 in \cite{Sengupta1985a}: 
To this end, suppose one aims to minimize the risk of a
portfolio. Moreover, one only has $n$ risky assets $A_i$
available. Let these assets provide a revenue $r_i$ in case of an
event $X_i$ and $0$ otherwise, i.e. $A_i=r_i \mathbbm{1}_{X_i}$ and
let the $A_i$ be independently, identically distributed with
probability $\P\in \Pcal$, where $\Pcal$ denotes a pre-defined
ambiguity set as described in Section \ref{Sec:introduction}. Assume that the covariance matrix of the assets $A_i$ is dominated by a matrix $\Sigma$, i.e., $0\preceq \text{Var}(A)\preceq \Sigma$ and we ask for a guaranteed revenue $w$ of our portfolio.

Then, the mean-variance model reads:
$$\min_x x^\top \Sigma x: \min_{\P\in \Pcal} \E_{\P}\left(\sum_{i=1}^n x_iA_i\right) \geq w, \sum_{i=1}^n x_i=1, x\geq 0,$$
which for i.i.d. assets $A_i$ is equivalent to
$$-\max_x -\sum_{i=1}^n \sigma_i x_i^2:\ \min_{\P\in \Pcal} \sum_{i=1}^n x_ir_i\P(X_i) \geq w, \sum_{i=1}^n x_i=1, x\geq 0.$$
This is indeed a special case of \eqref{Prob: Case1} since nonnegative
$\sigma_i,x_i$ lead to a concave objective
function and 
$\left\{x\in \R^n: \sum_{i=1}^n x_i=1, x\geq 0\right\}$ denotes a convex set for which e.g. the methods from \cite{Wiesemann2014a} can be applied. Thus, although addressing Case 1 as well, in the present article we focus on the following case, where we consider the sets $X_i$ as decision variables.

\textbf{Case 2:} Suppose the coefficients $x_i$ are given
parameters and the sets $X_i=[x_i^-,x_i^+]\subseteq
\R^m$ determine hypercubes. Consider the
boundaries of these hypercubes as decision variables. In addition, we assume w.l.o.g. that $X_i\subseteq T$ for well-posedness of $\mathbbm{P}(X_i)$ and additionally assume a linear objective function for ease of presentation. In particular, we consider:

\begin{subequations}\label{Prob: Case2}
	\begin{align}
		\max_{((x^-)^\top,(x^+)^\top)\in C}\ & \sum_{i=1}^k\sum_{j=1}^m c^-_{ij}x^-_{ij}+c^+_{ij}x^+_{ij}: \\
		\st\ & b \leq \min_{\mathbbm{P}\in \Pcal} \sum_{i=1}^k x_i \mathbbm{P}([x_i^-,x_i^+]),\label{Constr: Case2}
	\end{align}
\end{subequations}
where $C\subseteq \R^{2km}$ denotes a polytope of $n=2k$ decision vectors, each of dimension $m$. Note, that Case 2 appears to be more challenging than Case 1 as the function
$$v(x^-,x^+,t)\coloneqq \sum_{i=1}^k x_i\mathbbm{1}_{[x_i^-,x_i^+]}(t)$$
is not only nonconvex in $t$ but also in $((x_i^-)^\top,(x_i^+)^\top)$. Despite of this mathematical challenge, this case already covers interesting applications in chemical separation processes as is illustrated in \cite{Dienstbier2023a}. 
%

Let us now introduce essential notation and concepts. We
refer to \cite{Barvinok2002a} and \cite{Shapiro2000a}
for more information. The main challenges in Problems
\eqref{Prob: Case1} and \eqref{Prob: Case2} arise from the DRO constraints \eqref{Constr: Case1} and \eqref{Constr: Case2}, since these constraints cannot be formulated with the canonic Euclidean inner product. Consequently, standard
reformulation arguments from robust optimization such as replacing the inner
adversarial optimization problem by the feasible region of its
dual and solve the resulting model as a
standard finite-dimensional convex problem, do not apply. However, the following inner product, illustrated in Section III.3.2 in \cite{Barvinok2002a}, allows a similar reformulation of \eqref{Constr: Case1} and \eqref{Constr: Case2}:

Let $\P$ denote a probability measure on the compact domain $T$ that is defined by a probability density $\uRTD{}{t}$, i.e. $d\P = \uRTD{}{t}dt.$
According to Riesz-Markov-Kakutani representation theorem $\P$ is unique, i.e. it is the only solution that satisfies $I(f)=\int f d\P$
for the linear functional $I:\Ccal(T)\rightarrow \R$ defined by $I(f)\coloneqq\int_0^T f(t)\uRTD{}{t}dt.$
The corresponding inner product
$$\langle f, \P\rangle \coloneqq \int_T f d\P$$
then constitutes a \emph{duality}, i.e. a non-degenerate inner
product. Moreover, this duality is more generally defined on
\emph{signed Radon measures}, denoted by $\Mcal(T)$. 

Suppose, we know a continuous approximation of the indicator function $\mathbbm{1}_{X_i}$, denoted by $\mathbbm{1}_{X_i}^c$. Then, we observe that the above product $\langle \cdot , \cdot \rangle: \Ccal(T)\times \Mcal(T)\rightarrow \R$, enables us to approximate \eqref{Eq: DRO_constraint_Prob} via the function $\sum_{i=1}^k x_i\mathbbm{1}_{X_i}\in \Ccal(T)$ and the probability measure $\P\in \Mcal(T)$ as follows:
\begin{subequations}\label{Prob: DRO_with_undefined_ambiguity_set}
	\begin{align}
		b \le \min_{\P}~&  \langle \sum_{i=1}^k x_i\mathbbm{1}_{X_i}^c,\mass{}\rangle & \\
		\text{s.t.}~& \mass{} \in \mathcal{M}(T)_{\ge 0}\label{Constr: DRO_with_undefined_ambiguity_set_1}\\
		&\langle 1, \mass{} \rangle \ge 1 \label{Constr: DRO_with_undefined_ambiguity_set_norm_1}\\
		&\langle -1, \mass{} \rangle \ge -1, \label{Constr: DRO_with_undefined_ambiguity_set_norm_2}
	\end{align}
\end{subequations}
where $\Mcal(T)_{\geq 0}$ denotes the cone of nonnegative Radon measures. Furthermore, Constraints \eqref{Constr: DRO_with_undefined_ambiguity_set_1} -- \eqref{Constr: DRO_with_undefined_ambiguity_set_norm_2} require $\mass{}$ to be a probability measure.

\subsection{Strenghtening DRO Models by Moment Control and Confidence Sets}

One of the major challenges in distributional robustness consists in
choosing additional constraints for \eqref{Prob: DRO_with_undefined_ambiguity_set} is on the one hand algorithmically tractable, but on the other hand also large enough to protect the solutions $x$ (in Case 1) and $x^-,x^+$ (in Case 2) against all \emph{realistic} uncertainties. Moreover, one aims to avoid including unrealistic uncertainties as those render the decisions $x$ and $x^-,x^+$ too conservative. 
Within our setting, it is also possible to add additional information
on the uncertain probability distributions. This leads to additional constraints that can be added to \eqref{Prob: DRO_with_undefined_ambiguity_set} while maintaining algorithmic tractability. 

First, we aim at bounding the \emph{first moment}, i.e. the expectation $\E_{\mass{}}(t)$, of $\mass{}$. The authors in \cite{Parys2015a} and other sources assume perfect knowledge about the first moment, whereas the authors of \cite{Delage2010a} only assume that the first moment is contained in an ellipsoid. In this article, we follow the latter modeling and assume that an estimate of the correct expectation $\EW{}$ and covariance matrix $\Sigma$ is known. Moreover, we assume, that the ellipsoidal uncertainty set containing $\E_{\mass{}}(t)$ is shaped by $\EW{}$, $\Sigma$ and a third parameter $\varepsilon_{\EW{}}>0$, that determines its size. The ellipsoidal uncertainty set is then given by $$\varepsilon_{\EW{}}-(\E_{\mass{}}(t)-\EW{})^\top \Sigma (\E_{\mass{}}(t)-\EW{})\geq 0, \Sigma\succeq 0.$$
In order to reformulate the above constraint by means of an inner product $\langle \cdot , \mass{}\rangle$, we apply Schur's complement and obtain the following equivalent SDP constraint, which fits the setting in \eqref{Prob: DRO_with_undefined_ambiguity_set}:
\begin{equation}\label{Eq: Sec2_first_moment}
	\left\langle \begin{bmatrix}
		\Sigma & t-\EW{}\\
		(t-\EW{})^\top & \varepsilon_{\EW{}}
	\end{bmatrix}, \mass{}\right\rangle \succeq 0.
\end{equation}

Similarly, one may assume that the underlying uncertain probability measure is given by a monomodal density function, see
e.g. \cite{Parys2015a}. Computationally, this assumption has the advantage, that, if $\Pcal$ contains
monomodal distributions with fixed first and second moments, \eqref{Prob: DRO_with_undefined_ambiguity_set} can be reformulated as an SDP. This is one of the main results in \cite{Parys2015a}. However, the corresponding SDP is not easy to incorporate into either \eqref{Prob: Case1} or \eqref{Prob: Case2} as it generally leads to bilinear terms and thereby intractable counterparts for both \eqref{Prob: Case1} and \eqref{Prob: Case2}. In particular, \cite{Rahimian2019a} state, that "with the current state of literature, monomodality cannot be modeled in a tractable manner". To circumvent this obstacle, we exploit the fact that monomodal distributions tend to have a relatively small variance. Thus, similar again to \cite{Delage2010a} in addition to the bounds on the first moment, we impose an upper bound on the \emph{second moment} as follows  
\begin{equation}\label{Eq: Sec2_second_moment}
	\langle -(t-\EW{})(t-\EW{})^\top ,\mass{}\rangle \succeq -\varepsilon_\Sigma\Sigma
\end{equation}
or, equivalently $\text{Var}_{\mass{}}(t)\preceq \varepsilon_\Sigma \Sigma$. Here, $\varepsilon_\Sigma\geq 1$ measures the maximum deviation of the covariance matrix compared to its estimate $\Sigma$.

Finally, we add \emph{confidence set} constraints, see e.g. \cite{Wiesemann2014a}, where we restrict the probability of certain subsets $T_i\subseteq T$, i.e.,
\begin{equation}\label{Eq: Sec2_confidence_sets}
	\langle \sign(\varepsilon_i)\mathbbm{1}_{T_i}^c(t), \mass{} \rangle \ge \varepsilon_i \text{ for every } i\in I.
\end{equation}

Note, that these constraints give us a lot of modeling power as we can model $\mass{}(T_i)\geq \varepsilon_i$ with $\varepsilon_i>0$ and $\mass{}(T_i)\leq -\varepsilon_i$ with $\varepsilon_i<0$. In particular, the normalization constraints \eqref{Constr: DRO_with_undefined_ambiguity_set_norm_1} and \eqref{Constr: DRO_with_undefined_ambiguity_set_norm_2} fall in this framework and will be modeled by setting $T_1=T_2=T$ and $\varepsilon_1 = -1, \varepsilon_2=1$ throughout the remainder of this article.

\subsection{Relation to the Literature}
In the existing literature, distributionally robust constraints are often encoded with the expectation $\E_{\P}(v(x,t))$, which in the present paper encodes the expectation of a nonconvex, in our case a piecewise-constant, function $v$ in $t$ by $\E_{\P}(v(x,t)) = \sum_{i=1}^k x_i \mathbbm{P}(X_i)$. Dropping the convexity assumption poses a stark contrast to the results in \cite{Wiesemann2014a} and \cite{Delage2010a}, where the underlying function $v(x,t)$ has to be both, convex and piecewise-affine in $x$ and $t$, see Condition (C3) in \cite{Wiesemann2014a} and Assumption 2 in \cite{Delage2010a}. However, \cite{Wiesemann2014a} and \cite{Xu2017a} present exceptions to these assumptions for specific cases, namely a very low number $|I|$ of confidence sets, see Observation 1ff in the electronic compendium of \cite{Wiesemann2014a} or even $|I|=0$ (\cite{Xu2017a}). As we consider indicator functions $\mathbbm{1}_{X_i}(t)$, that generally do not satisfy any of those assumptions, we attempt to extend the existing literature to nonconvex functions $v$. Moreover, in contrast to \cite{Dienstbier2023a}, we allow $T$ to be multivariate and consider simple functions $\sum_{i=1}^k x_i\mathbbm{1}_{[x_i^-,x_i^+]}(t)$ instead of either sole indicator functions with $k=1$ or simple functions with the simplifying assumption, that the $m$ entries of $t$ are independent. This increased generality is achieved at the cost of a potentially worse approximation accuracy.

Lastly, we briefly mention the differences of our approach to discrepancy-based DRO models that require an estimator for the true probability distribution $\hat{\rho}$ and restrict $\Pcal$ based on a given metric, e.g. the Wasserstein metric. Here, given an estimated $\hat{\rho}$, these ambiguity sets are formed of all probability distributions, that originate from $\hat{\rho}$ by transferring at most a given probability mass. We refer to the excellent review \cite{Rahimian2019a} for further details.

\section{Distributionally robust constraints dependent on simple functions}\label{Sec: DRO_indicator_functions}

For both, Cases 1 and 2 from Section \ref{Sec: Problem Setting}, we consider the DRO constraint \eqref{Prob: DRO_with_undefined_ambiguity_set} where $\Pcal$ is defined by \eqref{Eq: Sec2_first_moment}, \eqref{Eq: Sec2_second_moment} and \eqref{Eq: Sec2_confidence_sets}.
To this end, let again $b\in \R$, $T\subseteq
\R^m$ be a compact set, and $I\subseteq \N$ denote a finite index set. Next we define the considered ambiguity set. We assume
a 'typical', i.e., nominal distribution with mean $\EW{} \in \R^m$
and covariance matrix $\Sigma \in \R^{m\times m}$ is given, for example from
expert knowledge or by estimation from given data. In formulas, we consider

\begin{subequations}
	\label{Prob: Primal_Purity_Constraint_indicator}
	\begin{align}
		b \leq \min_{\mass{}}~& \langle \sum_{i=1}^k x_i \mathbbm{1}_{X_i}^c,\mass{}\rangle && \label{Constr: Objective_Primal_indicator}\\
		\text{s.t.}~& \mass{} \in \mathcal{M}(T)_{\ge 0} \\
		& \langle \begin{bmatrix}
			\Sigma & t-\mu\\
			(t-\EW{})^\top & \varepsilon_{\EW{}}
		\end{bmatrix}, \mass{}\rangle \succeq 0 \label{Constr: First_Moment_indicator} \\
		&\langle -(t-\EW{})(t-\EW{})^\top ,\mass{}\rangle \succeq -\varepsilon_\Sigma\Sigma  && \label{Constr: Second_Moment_indicator}\\
		&\langle \sign(\varepsilon_i)\mathbbm{1}_{T_i}^c(t), \mass{} \rangle \ge \varepsilon_i && i\in I \label{Constr: Primal_indicator_confidence_set},
	\end{align}
\end{subequations}
where a choice of $T_1=T, \varepsilon_1=-1$ and $T_2=T,
\varepsilon_2=1$ implies that $\mass{}(T)=1$, i.e. $\mass{}$ is a
probability measure on $T$. In the following, we aim at deriving an
algorithmically tractable
reformulation of this set of constraints. We note that in order to
dualize \eqref{Prob: Primal_Purity_Constraint_indicator}, we consider
continuous approximators $x_i \mathbbm{1}_{X_i}^c,
\sign(\varepsilon_i)\mathbbm{1}_{T_i}^c$ of the indicator functions
$x_i \mathbbm{1}_{X_i}, \sign(\varepsilon_i)\mathbbm{1}_{T_i}$. The
existence of approximators that are arbitrarily close to the indicator
functions is given by the seminal Lemma of Urysohn, see e.g. \cite{munkrestopology}. In particular, we choose $\mathbbm{1}^c_{X_i} \geq \mathbbm{1}_{X_i}$, an upper approximator whenever $x_i\geq 0$ and a lower approximator whenever $x_i<0$. The opposite approximators are chosen for $\mathbbm{1}_{T_i}$, i.e., we choose $\mathbbm{1}_{T_i}^c\leq \mathbbm{1}_{T_i}$ if $\varepsilon_i\geq 0$ and $\mathbbm{1}_{T_i}^c\geq \mathbbm{1}_{T_i}$ whenever $\varepsilon_i<0$. This establishes the following key property
\begin{equation}\label{Eq: Urysohn_approx}
	x_i\mathbbm{1}^c_{X_i} \geq x_i\mathbbm{1}_{X_i}\text{ and } \sign(\varepsilon_i)\mathbbm{1}^c_{T_i} \leq \sign(\varepsilon_i)\mathbbm{1}_{T_i}.
\end{equation}

In the following, we will define necessary ingredients for being
able to reformulate such a DRO constraint by dualizing
\eqref{Prob: Primal_Purity_Constraint_indicator}. Subsequently, a tractable and high-quality inner approximation of the resulting
constraint will be obtained.
We first employ duality theory using an adjoint operator:

\begin{remark}\label{Rem: existence_adjoint_operator}
Let $\Scal^r$ denote the set of symmetric $r$ by $r$ matrices. It might not be immediately clear whether an adjoint operator with respect to the primal operator $\mathcal{A} : \Mcal(T) \rightarrow \Scal^{m+1}\times \Scal^m \times \R^I$ of \eqref{Prob: Primal_Purity_Constraint_indicator} exists. 
However, it is constructed in a quite straightforward manner:
First, we observe that for the inner products containing matrices $A\in \Scal^r$, we have
$$\langle \langle A, \mass{} \rangle , Y \rangle_F = \langle \langle A, Y\rangle_F , \mass{} \rangle \text{ for arbitrary }\mass{}\in \Mcal(T),Y\in \Scal^r,$$
where, $\langle \cdot,\cdot\rangle_F: \Scal^r\times \Scal^r\rightarrow \R$ denotes the Frobenius inner product. In particular, for $r\in \{m,m+1\}$, this includes the matrices
$$A\in \left\{\begin{bmatrix}
	\Sigma & t-\EW{}\\
	(t-\EW{})^\top & \varepsilon_{\EW{}}
\end{bmatrix}, -(t-\EW{})(t-\EW{})^\top \right\}.$$

For the inner products containing only the entries $\sign(\varepsilon_i) \mathbbm{1}_{T_i}^c$ of $\mathcal{A}$, we have 
$$\langle \sign(\varepsilon_i) \mathbbm{1}_{T_i}^c, \mass{} \rangle y = \langle \sign(\varepsilon_i) \mathbbm{1}_{T_i}^c y , \mass{} \rangle \text{ for every }\mass{}\in \Mcal(T), y\in \R.$$
Hence, we have constructed an adjoint operator $\mathcal{B}: \Scal^{m+1}\times \Scal^m\times \R^I\rightarrow \Ccal(T)$ to $\mathcal{A}$, which is defined by 
$$\left\langle \begin{bmatrix}
	\Sigma & t-\EW{}\\
	(t-\EW{})^\top & \varepsilon_{\EW{}}
\end{bmatrix}, Y_1\right\rangle + \langle -(t-\EW{})(t-\EW{})^\top, Y_2\rangle + \sum_{i\in I}\sign(\varepsilon_i) \mathbbm{1}_{T_i}^c y_i.$$
Moreover, $\mathcal{B}$ is unique due to Riesz' representation
theorem, see e.g. \cite{Brezis2010a}.
\end{remark}

With this adjoint operator, we derive the following dual program for \eqref{Prob: Primal_Purity_Constraint_indicator}:

\begin{subequations}
	\label{Prob: Dual_Purity_Constraint_indicator}
	\begin{align}
		b \le \max_{y_i,Y_1,Y_2}~& \sum_{i\in I} \varepsilon_i y_i -\varepsilon_\Sigma \langle \Sigma, Y_2 \rangle \label{Obj: Objective_Dual_indicator} \\
		\text{s.t.}~& \sum_{i=1}^k x_i \mathbbm{1}_{X_i}^c - \left\langle \begin{bmatrix}
			\Sigma & t-\EW{}\\
			(t-\EW{})^\top & \varepsilon_{\EW{}}
		\end{bmatrix} , Y_1 \right\rangle -\langle -(t-\EW{})(t-\EW{})^\top , Y_2\rangle \notag \\
		& -\sum_{i\in I} \sign(\varepsilon_i) \mathbbm{1}_{T_i}^c y_i \in \Ccal(T)_{\ge 0} \label{Constr: Dual_indicator}\\
		& Y_1 \in \Scal^{m+1}_{\succeq 0}, Y_2 \in \Scal^m_{\succeq 0}, y \in \R_{\geq 0}^I,
	\end{align}
\end{subequations}
where $\Ccal(T)_{\ge 0}$ denotes the cone of the continuous,
nonnegative functions on $T$.

As usual in reformulation
approaches in robust optimization, we aim to apply strong duality. 
Indeed, next we establish strong duality between \eqref{Prob:
  Primal_Purity_Constraint_indicator} and \eqref{Prob:
  Dual_Purity_Constraint_indicator} that can be seen as a direct
corollary of Corollary 3.0.2 in \cite{Shapiro2000a} or as a direct
consequence of the dualization theory illustrated, e.g. in \cite{Barvinok2002a}.

\begin{theorem}\label{Thm: strong_duality}
	Suppose that $\mass{} \sim \Ncal(\EW{},\Sigma)$ is both, a strictly positive Radon measure and feasible for \eqref{Prob: Primal_Purity_Constraint_indicator}. Then, the duality gap of the problems \eqref{Prob: Primal_Purity_Constraint_indicator} and \eqref{Prob: Dual_Purity_Constraint_indicator} is zero.
\end{theorem}

\begin{proof}
	We observe that $\mass{} \sim \Ncal(\EW{},\Sigma)$ is feasible for \eqref{Prob: Primal_Purity_Constraint_indicator}, i.e. \eqref{Prob: Primal_Purity_Constraint_indicator} is "consistent" in the definition of Shapiro. Furthermore, $T$ is compact and the functions in the objective as well as in the constraints of \eqref{Prob: Primal_Purity_Constraint_indicator} are continuous. Due to the isometry of the metric spaces $(\Scal^r, \langle\cdot , \cdot \rangle_F)$ and $(\R^{\frac{r(r-1)}{2}}, \langle \cdot , \cdot \rangle)$, we further reformulate \eqref{Prob: Primal_Purity_Constraint_indicator} as a conic program with $\mathcal{A}\mass{} -b\in K$, where the cone $K\subseteq \R^{(m+1)m/2 + m(m-1)/2+|I|}$.
	Hence, strong duality follows from Corollary 3.1 in \cite{Shapiro2000a}.
\end{proof}

\subsection{Computation of feasible solutions by a discretized robust counterpart}\label{Subsec: discretized_DRO_reformulation_indicator}

In this section, we derive an algorithmically tractable model
for the robust counterpart \eqref{Prob:
Dual_Purity_Constraint_indicator}.
A standard approach to find an approximate solution to this
semiinfinite (SIP) problem 
  is to sample the semiinfinite constraint \eqref{Constr:
    Dual_indicator} and solve the resulting finite-dimensional SDP
  that only contains the sampled constraints. However, a
feasible solution to a finite subsets of the constraints in
\eqref{Constr: Dual_indicator} does not necesarily satisfy
\eqref{Constr: Dual_indicator} itself. 
This means that the obtained solution may not satisfy \eqref{Prob:
  Dual_Purity_Constraint_indicator} and thus by solving Case 1 or 2
with respect to this relaxation of \eqref{Prob:
  Dual_Purity_Constraint_indicator}, we might obtain a solution,
which is not necessarily protected against the uncertainties in the
ambiguity set $\Pcal$, i.e. is not robust and does not necessarily satisfy \eqref{Prob: Primal_Purity_Constraint_indicator}. 

In this work, we however aim for a robust constraint for $\Pcal$ as
for many applications a guaranteed protection is important,
e.g. in medical applications.

To this end, we propose a discretization scheme that
provides an inner approximation of \eqref{Constr:
  Dual_indicator}. This means that every solution of the discretization of \eqref{Prob: Dual_Purity_Constraint_indicator} will
indeed satisfy \eqref{Prob: Dual_Purity_Constraint_indicator} and
thereby guarantee that the corresponding decision variables
$x_i$ for Case 1 and $x_i^-,x_i^+$ for Case 2 are feasible for \eqref{Prob: Primal_Purity_Constraint_indicator}. This robust formulation will make use of Lipschitz continuity of the non-indicator functions in
\eqref{Constr: Dual_indicator}, i.e., the Lipschitz continuity of the polynomial
$$p_Y(t)\coloneqq \left\langle \begin{bmatrix}
	\Sigma & t-\EW{}\\
	(t-\EW{})^\top & \varepsilon_{\EW{}}
\end{bmatrix} , Y_1 \right\rangle +\langle (t-\EW{})(t-\EW{})^\top , Y_2\rangle.$$
In fact, the polynomial $p_Y$ is Lipschitz continuous since $T$ is compact and its coefficients $Y_1,Y_2$ are bounded: 

\begin{lemma}\label{Lemma: Lipschitz_continuity_of_p_indicator}
	Let $T_1=T$ and $\varepsilon_1=-1$. Furthermore, for every $i\in I\setminus\{1\}$ let $\EW{} \in T_i$ if $\varepsilon_i>0$ and $\EW{} \notin T_i$ if $\varepsilon_i<0$. Then, the polynomial $p_Y(t)$ is Lipschitz continuous in $t$ with a uniform Lipschitz constant $L$.
\end{lemma}

\begin{proof}
	Due to the compactness of $T$, it suffices to show that for every feasible solution of \eqref{Prob: Dual_Purity_Constraint_indicator} the entries $Y_1,Y_2$ are bounded. To this end, let $\varepsilon_1=-1, \varepsilon_2=1$. In addition, we assume w.l.o.g. $\varepsilon_i >0$ for every $i\in I \setminus\{1\}$. This is due to the fact that every constraint
	$$	\langle \sign(\varepsilon_i)\mathbbm{1}_{T_i}^c,\P \rangle \geq \varepsilon_i \text{ with } -1 \leq \varepsilon_i < 0$$
	can equivalently be expressed by
	$$\langle \mathbbm{1}_{T_i^C}^c,\P \rangle \geq 1+\varepsilon_i.$$
	In order to prove this equivalence, we note that $\sign(\varepsilon_i)=-1$, add $1$ on both sides and consider the complement $T_i^C$ of $T_i$. 
	
	Now, we first prove that $\Tr(Y_1)<\infty$:
	Let $t=\EW{}$ and $v_i$ being the eigenvectors and $\lambda_i$ the eigenvalues of $Y_1$ then \eqref{Constr: Dual_indicator} implies:
	\begin{equation}\label{Eq: help3}
		\begin{aligned}
			\lambda_{\min}\left(\begin{bmatrix}
				\Sigma & 0 \\ 0 & \varepsilon_{\EW{}}
			\end{bmatrix}\right) \Tr(Y_1) & = \sum_{i=1}^{m+1} \lambda_i  \lambda_{\min}\left(\begin{bmatrix}
				\Sigma & 0 \\ 0 & \varepsilon_{\EW{}}
			\end{bmatrix}\right) \overset{*}{\leq} \sum_{i=1}^{m+1} \lambda_i v_i^\top \begin{bmatrix}
				\Sigma & 0 \\ 0 & \varepsilon_{\EW{}}
			\end{bmatrix} v_i\\
			& \leq \left\langle \begin{bmatrix}
				\Sigma & 0 \\ 0 & \varepsilon_{\EW{}}
			\end{bmatrix}, Y_1 \right\rangle \overset{{\eqref{Constr: Dual_indicator}}}{\leq} \sum_{i=1}^k x_i \mathbbm{1}_{X_i}^c(\EW{}) - \sum_{i\in I} \sign(\varepsilon_i) y_i,
		\end{aligned}
	\end{equation}
	where (*) holds due to the Rayleigh-Ritz principle, see e.g. \cite{Brezis2010a} for further details. We show that \eqref{Eq: help3} is bounded from above for every feasible solution to \eqref{Prob: Dual_Purity_Constraint_indicator} by considering the following LP:
	\begin{equation}\label{Eq: help4}
		\min_{y\in \R^I_{\geq 0}} \sum_{i\in I} \sign(\varepsilon_i) \mathbbm{1}_{T_i}^c(\EW{})y_i:\ \sum_{i\in I}\varepsilon_i y_i \geq b, 
	\end{equation}
	whose constraint can be derived from \eqref{Obj: Objective_Dual_indicator} and the fact that both $\Sigma$ and $Y_2$ are positive semidefinite. Moreover, this is equivalent to
	\begin{equation*}
		\min_{y\in \R^I_{\geq 0}} -y_1+\sum_{i\in I \setminus \{1\}} y_i:\ \sum_{i\in I}\varepsilon_i y_i \geq b.
	\end{equation*}
	due to $\EW{}\in T_i$ for every $i\in I$. Furthermore, it is bounded from below by $0$ since its dual LP:
	\begin{align*}
		\max_{z\geq 0} b z : - z &\leq -1,\\ 
		\varepsilon_i z &\leq 1 &&\text{ for every } i\in I\setminus\{1\},
	\end{align*}
	is feasible for $z=1$ since w.l.o.g. $|\varepsilon_i|\leq 1$. Consequently, this provides a lower bound of $b$ to \eqref{Eq: help4} and thereby an upper bound to $\text{Tr}(Y_1)$ via \eqref{Eq: help3}.
	
	Let $\lambda_{\min}(\Sigma)> 0$ denote the minimal eigenvalue of $\Sigma$ and $\lambda_i$ the eigenvalues of $Y_2$ with respect to eigenvector $v_i$. Then, on the one hand, we have
	\begin{equation}\label{Eq: help1}
		\begin{aligned}
			\varepsilon_\Sigma \lambda_{\min}(\Sigma) \Tr(Y_2) & = \varepsilon_\Sigma \sum_{i=1}^m \lambda_i \lambda_{\min}(\Sigma) \overset{(*)}{\leq} \varepsilon_\Sigma \sum_{i=1}^m \lambda_i v_i^\top \Sigma v_i = \varepsilon_\Sigma \left\langle \Sigma , \sum_{i=1}^m \lambda_i v_iv_i^\top \right\rangle\\
			& = \varepsilon_\Sigma \langle \Sigma , Y_2 \rangle \overset{\eqref{Obj: Objective_Dual_indicator}}{\leq} \sum_{i\in I} \varepsilon_i y_i
		\end{aligned}
	\end{equation}	
	
	where (*) holds because of the Rayleigh-Ritz principle. In order to show that \eqref{Eq: help1} is bounded, we show that the following linear program is bounded from above:
	\begin{equation}\label{Eq: LP_help}
		\max_{y\in \R^I_{\geq 0}} \varepsilon^\top y:\  \tau^\top y \leq \sum_{i=1}^k x_i \mathbbm{1}_{X_i}^c(\EW{}), 
	\end{equation}
	where $\tau_i = \sign(\varepsilon_i) \mathbbm{1}_{T_i}(\EW{}).$  Note that $\tau\neq 0$ due to $\EW{} \in T_2$. Similar as before, the constraint in \eqref{Eq: LP_help} can be derived from \eqref{Constr: Dual_indicator} with $t=\EW{}$ in the following way:
	\begin{equation}\label{Eq: help2}
		\begin{aligned}
			\sum_{i=1}^k x_i \mathbbm{1}_{X_i}^c(\EW{}) & \geq \sum_{i=1}^k x_i \mathbbm{1}_{X_i}^c(\EW{}) - \langle \begin{bmatrix}
				\Sigma & 0 \\ 0 & \varepsilon_{\EW{}}
			\end{bmatrix}, Y_1 \rangle \geq \sum_{i\in I} \sign(\varepsilon_i) \mathbbm{1}_{T_i}^c(\EW{})y_i
		\end{aligned}
	\end{equation}
	Then, weak duality implies
	\begin{equation}
		\eqref{Eq: LP_help} \leq \min_{z\in \R_{\geq 0}} z \sum_{i=1}^k x_i \mathbbm{1}_{X_i}^c(\EW{}):\ z \tau -\varepsilon \geq 0. 
	\end{equation}
	Observe that $z=1$ is a feasible solution since
	$$\tau_i=\sign(\varepsilon_i)\mathbbm{1}_{T_i}^c(\EW{})=1>\varepsilon_i$$
	for every $i\in I\setminus\{1\}$ and $\tau_1=-1=\varepsilon_1$. Thus, we obtain an upper bound for \eqref{Eq: LP_help} and thereby for $\text{Tr}(Y_2)$. Finally, we proved that the coefficients of $p_Y(t)$ are bounded and the claim follows. 
\end{proof}

Observe, that the assumptions on the confidence sets $T_i$, i.e., that either it is
$\EW{} \in T_i$ whenever $\varepsilon_i>0$ or $\EW{} \notin T_i$ if
$\varepsilon_i<0$, limits the power of modeling ambiguity sets $\Pcal$. Indeed, we our model does not include upper bounds on $\P(T_i)$ if $\mu \in T_i$ and lower bounds if $\mu \notin T_i$. We note, that this limitation is rather mild as most real-world distributions are concentrated around their respective expectation to some degree. Consequently, since the requirement above still allows us to force the probability mass of $\mass{}\in \Pcal$ towards the estimated expected value $\EW{}$, it seems not very restrictive in practice. In fact, discrepancy based approaches such as Wasserstein balls yield a similar structure.

If confidence sets are used,
restrictions in modeling are fairly common, also for example in the so-called nesting condition in \cite{Wiesemann2014a} and the references therein. In addition, there are relevant settings where the assumption from the
above lemma can be
weakened. Indeed, in \cite{Dienstbier2023a} it is shown that for one-dimensional $T$, no such assumption is needed at all. 

In the following Lemma, we establish an inner approximation of the DRO constraint \eqref{Constr: Dual_indicator}. To this end, we denote by $T_N=\delta_N \mathbbm{Z}^m\cap T$ the standard lattice with stepsize $\delta_N \in \R_{>0}$, that serves as a discretization of $T$. Moreover, we define a \emph{level set} $L_h$ by 
$$L_h\coloneqq \left\{t\in T:\ \sum_{i=1}^k x_i \mathbbm{1}_{X_i}(t)-\sum_{i\in I} \sign(\varepsilon_i) \mathbbm{1}_{T_i}(t) =h\right\},$$
where $h$ denotes the \emph{height} of the specific level set. The motivation to consider these level sets is, that on the boundaries of $L_h$ the indicator functions $\mathbbm{1}_{X_i},\mathbbm{1}_{T_i}$ abruptly change and any potential Lipschitz constant $L$ for the continuous approximations $\mathbbm{1}_{X_i}^c,\mathbbm{1}_{T_i}^c$ of $\mathbbm{1}_{X_i},\mathbbm{1}_{T_i}$ tends to infinity, the closer the continuous approximation is. Consequently, an approximation of the left-hand side of \eqref{Constr: Dual_indicator} solely based on Lipschitz continuity may become quite poor.

Thus, we address the indicator functions separately. To this end, let us first denote 
\begin{align*}
	f^{c}(t)\coloneqq & \sum_{i=1}^k x_i \mathbbm{1}_{X_i}^{c}(t) - \left\langle \begin{bmatrix}
		\Sigma & t-\EW{}\\
		(t-\EW{})^\top & \varepsilon_{\EW{}}
	\end{bmatrix} , Y_1 \right\rangle +\langle (t-\EW{})(t-\EW{})^\top , Y_2\rangle\\ 
	&\quad  -\sum_{i\in I} \sign(\varepsilon_i) \mathbbm{1}_{T_i}^{c}(t) y_i
\end{align*}
for fixed $Y_1 \in \Scal^{m+1}_{\succeq 0}, Y_2 \in \Scal^m_{\succeq 0}, y \in \R_{\geq 0}^I$ and observe the equivalence
$$\eqref{Constr: Dual_indicator} \Leftrightarrow f^c(t)\geq 0 \text{ for every } t \in T.$$ 
Let us further observe, that in most applications, we can assume that $X_i\cap T_N\neq \emptyset$ and $T_i\cap T_N\neq \emptyset$, whenever $\delta_N$ is sufficiently small, e.g. if every $X_i$ and $T_i$ contains open sets. In particular, we assume that $\delta_N$ is chosen small enough, such that for every $t\in L_h$, we have that there is a $\bar{t}\in T_N\cap L_h$ with $\|t-\bar{t}\|\leq \sqrt{m}\delta_N$. Since $T_N = \delta_N \Z^m\cap T$, this guarantees that for every $t\in L_h$, there is a nearby sample point also contained in $L_h$.  Consequently, as seen in Lemma \ref{Lemma: Lipschitz_continuity_of_p_indicator}, we can address the differences on $f^c$ evaluated on sample points $\bar{t}\in T_N$ compared to the nearby non-sample points $t\in T\setminus T_N$ by exploiting Lipschitz continuity on the polynomial part $p$ of $f^c$.  Finally, we observe that the union of all these level sets $\bigcup_{h} L_h=T$ is a finite, disjoint decomposition of $T$ and thus, we have addressed all potential deviations of $f^c$ between values on $T\setminus T_N$ and $T_N$. To make these arguments precise:

\begin{lemma}\label{Lemma:inner_approx}
	Let $L>0$ be the Lipschitz constant of $p_Y$. Let further $\delta_N$ be sufficiently small, such that for every $t\in T$ with w.l.o.g. $t\in L_h$, there exists a $\bar{t}\in T_N\cap L_h$ with $\|t-\bar{t}\|\leq \delta_N\sqrt{m}$. Then, the finitely many constraints \begin{equation}\label{Eq: Dual_indicator_discretized}
		f(\bar{t})-L\delta_N\sqrt{m} \geq 0 \text{ for every } \bar{t}\in T_N
	\end{equation}
imply the semiinfinite constraint
$$f^c(t) \geq 0 \text{ for every } t\in T.$$
\end{lemma}

\begin{proof}
	We first suppose w.l.o.g. that $t\in L_h$. Then, there exists a $\bar{t}\in L_h$ such that $\|t-\bar{t}\|\leq \delta_N\sqrt{m}$ and hence 
	\begin{align*}
		f^c(t)+L\delta_N\sqrt{m} & \geq f^c(t)+L\|t-\bar{t}\|  \overset{(1)}{\geq} f^c(t) +|p_Y(\bar{t})-p_Y(t)|\\
		& \overset{\eqref{Obj: Objective_Dual_indicator}}{\geq} \sum_{i=1}^k x_i \mathbbm{1}_{X_i}^c(t) - \sum_{i\in I} \sign(\varepsilon_i) \mathbbm{1}_{T_i}^c(t) + p_Y(\bar{t}) \\
		& \overset{(2)}{\geq} \sum_{i=1}^k x_i \mathbbm{1}_{X_i}(t) - \sum_{i\in I} \sign(\varepsilon_i) \mathbbm{1}_{T_i}(t) + p_Y(\bar{t}) = f(\bar{t})\,
	\end{align*}
	where (1) holds due to definition of $L$ and (2) holds due to \eqref{Eq: Urysohn_approx}.
\end{proof}

Note, that Lemma \ref{Lemma:inner_approx} provides a sufficient criterion for the SIP constraint \eqref{Constr: Dual_indicator}. Thus, replacing \eqref{Constr: Dual_indicator} by \eqref{Eq: Dual_indicator_discretized} gives an inner approximation of \eqref{Prob: Dual_Purity_Constraint_indicator}. Therefore, the existence of $y,Y_1,Y_2$ satisfying \eqref{Eq: Dual_indicator_discretized} in addition to the remaining constraints of \eqref{Prob: Dual_Purity_Constraint_indicator} guarantees that the DRO constraint \eqref{Prob: Primal_Purity_Constraint_indicator} is satisfied.

\subsection{Tractable approximations for DRO}
We note that \eqref{Prob: Primal_Purity_Constraint_indicator} is
often considered as the (nonconvex) DRO constraint embedded in an otherwise convex program, e.g. as illustrated by Case 1 and 2 in Section \ref{Sec: Problem Setting}. Hence, instead of considering constant $x_i,X_i$, we investigate in the following
paragraphs how the Lemma \ref{Lemma:inner_approx} approximation can be
applied to Case 1, i.e. decision variables $x_i$ and Case 2, with
decision variables $x_i^-,x_i^+$ that define the box
$X_i=[x_i^-,x_i^+]$. For the sake of simplicity, we assume that the
objective of DRO is linear. However, the results below hold analogously for maximizing concave objective functions as well. For Case 1 let $x\in C \subseteq \R^n$ be a decision variable. We recall that $n=k$ and consider:
\begin{subequations}\label{Prob: primal_linear_x_i_Case1}
	\begin{align}
		\max_{x, Y_1,Y_2,y} & c^\top x \\
			\text{ s.t.}~ & \sum_{i\in I} \varepsilon_i y_i -\varepsilon_\Sigma \langle \Sigma, Y_2 \rangle \geq b \label{Constr: primal_linear_x_i_objective_constraint1} \\
			& \sum_{i=1}^k x_i \mathbbm{1}_{X_i}^c(t) - \left\langle \begin{bmatrix}
				\Sigma & t-\EW{}\\
				(t-\EW{})^\top & \varepsilon_1
			\end{bmatrix} , Y_1 \right\rangle  \notag \\
			& +\left\langle (t-\EW{})(t-\EW{})^\top , Y_2\right\rangle -\sum_{i\in I} \sign(\varepsilon_i) \mathbbm{1}_{T_i}^c(t) y_i\geq 0 \qquad \forall t\in T \label{Constr: primal_linear_x_i_objective_constraint2}\\
			& x\in C, Y_1 \in \Scal^{m+1}_{\succeq 0}, Y_2 \in \Scal^m_{\succeq 0}, y \in \R_{\geq 0}^I.
	\end{align}
\end{subequations}

It turns out that computing lower bounds for \eqref{Prob: primal_linear_x_i_Case1} is tractable:
 
\begin{theorem}\label{Thm: SDP_for_indicator_with_varying_x_i}
	A solution to the following semidefinite problem yields a feasible solution to the semiinfinite problem \eqref{Prob: primal_linear_x_i_Case1}.
	\begin{subequations} \label{Prob: discretized_indicator_linear_obj}
		\begin{align}
			\max_{x,Y_1,Y_2,y}\; & c^\top x\\
			 \text{ s.t.}~ & \sum_{i\in I} \varepsilon_i y_i -\varepsilon_\Sigma \langle \Sigma, Y_2 \rangle \geq b  \\
			& \sum_{i=1}^k x_i \mathbbm{1}_{X_i}(\bar{t}) - \left\langle \begin{bmatrix}
				\Sigma & \bar{t}-\EW{}\\
				(\bar{t}-\EW{})^\top & \varepsilon_1
			\end{bmatrix} , Y_1 \right\rangle +\langle (\bar{t}-\EW{})(\bar{t}-\EW{})^\top , Y_2\rangle \notag \\
			& -\sum_{i\in I} \sign(\varepsilon_i) \mathbbm{1}_{T_i}(\bar{t}) y_i - L\delta_N\sqrt{m}\geq 0 \qquad\qquad\qquad\qquad\qquad \forall \bar{t}\in T_N \label{Constr: discretized_purity_indicator_linear_obj}\\
			& x\in C, Y_1 \in \Scal^{m+1}_{\succeq 0}, Y_2 \in \Scal^m_{\succeq 0}, y \in \R_{\geq 0}^I.
		\end{align}
	\end{subequations}
\end{theorem}

\begin{proof}
	Given an arbitrary $x\in C$. Due to Lemma \ref{Lemma:inner_approx}, we observe that Constraint \eqref{Constr: discretized_purity_indicator_linear_obj} implies $f^c(t)\geq 0$ for every $t\in T$, i.e. \eqref{Constr: primal_linear_x_i_objective_constraint2}. Hence, the claim follows.
\end{proof}

We note that $\sum_{i=1}^k x_i \mathbbm{1}_{X_i}$
is linear and thus convex in the $x_i$. Thus, if the number
of confidence sets $|I|$ is low, Problem \eqref{Prob:
  discretized_indicator_linear_obj} satisfies the (weakened)
conditions needed for Theorem 1 in \cite{Wiesemann2014a} and can be
exactly reformulated as a convex program by applying their methods,
whereas the proposed method in this paper only provides a lower bound
on \eqref{Prob: primal_linear_x_i_Case1}. However, our approach can also be
used for a large number of confidence sets. In addition, it does not
depend on convexity and can also be used in nonconvex settings. This can be seen by the following result for Case 2, where $T=[0,M]^m$ and $X_i=[x_i^-,x_i^+]$ are supposed to be $k=2n$ hypercubes: 
\begin{subequations}\label{Prob: primal_linear_X_i}
	\begin{align}
		\max & \sum_{i=1}^k (c^-_i)^\top x_i^- + (c^+_i)^\top x_i^+\\
		\st~ & \sum_{i\in I} \varepsilon_i y_i -\varepsilon_\Sigma \langle \Sigma, Y_2 \rangle \geq b \label{Constr: primal_linear_X_i_objective_constraint1} \\
		& \sum_{i=1}^k x_i \mathbbm{1}_{[x_i^-,x_i^+]}^c(t) - \left\langle \begin{bmatrix}
			\Sigma & t-\EW{}\\
			(t-\EW{})^\top & \varepsilon_1
		\end{bmatrix} , Y_1 \right\rangle \notag \\
		& \qquad +\left\langle (t-\EW{})(t-\EW{})^\top , Y_2\right\rangle -\sum_{i\in I} \sign(\varepsilon_i) \mathbbm{1}_{T_i}^c(t) y_i\geq 0 && \forall t\in T \label{Constr: primal_linear_X_i_objective_constraint2}\\
		& x_i^-,x_i^+\in C, Y_1 \in \Scal^{m+1}_{\succeq 0}, Y_2 \in \Scal^m_{\succeq 0}, y \in \R_{\geq 0}^I.
	\end{align}
\end{subequations}
We note, that $\sum_{i=1}^k x_i \mathbbm{1}_{[x_i^-,x_i^+]}^c$ is nonconvex in the variables $x_i^-,x_i^+\in \R^m$.
In the following theorem, we model the indicator function
$\mathbbm{1}_{[x_i^-,x_i^+]}^c:T_N\rightarrow \R$ by binary variables
$\tilde{b}_{\bar{t}}^i$. Additionally, we ensure, that these variables properly model $\mathbbm{1}_{[x_i^-,x_i^+]}^c(\bar{t})$ by tracking the "jumps" from $0$ to $1$
at $x_{ij}^-$ in direction $j\in [m]$ by additional binary variables
$\Delta_{\bar{t}}^{-,i,j}$ and the "jumps" form $1$ to $0$ at $x_{ij}^+$ in direction $j\in [m]$ by
$\Delta_{\bar{t}}^{+,i,j}$ respectively. For univariate simple functions a modeling approach along these lines was given in \cite{Dienstbier2020a} for an engineering application in
the design of particulate products.

\begin{theorem}\label{Thm: MIP_multidim}
	Let $M_\delta\coloneqq \{0,\delta_N,\ldots, M\}$ the discretization of $[0,M]$, $T_0^j=\{\bar{t}\in T_N: \bar{t}_j=0\}\subseteq T_N$ a set of boundary points of $T_N = \delta_N\Z^m\cap [0,M]^m$. Then, a solution to the following MISDP yields a feasible solution to \eqref{Prob: primal_linear_X_i}.

	\begin{subequations}\label{Prob: Dual_Purity_Constraint_indicator_discretized}
		\begin{align}
			\max &\sum_{i=1}^k (c^-_i)^\top x_i^- + (c^+_i)^\top x_i^+\\
			\st~& \sum_{i\in I} \varepsilon_i y_i -\varepsilon_\Sigma \langle \Sigma, Y_2 \rangle \geq b \label{Constr: discretized_dual_objective_indicator_geq0} \\
			& \sum_{i=1}^k x_i\tilde{b}_{\bar{t}}^i - \left\langle \begin{bmatrix}
				\Sigma & \bar{t}-\EW{}\\
				(\bar{t}-\EW{})^\top & \varepsilon_1
			\end{bmatrix} , Y_1 \right\rangle \notag\\
			& \qquad +\langle (\bar{t}-\EW{})(\bar{t}-\EW{})^\top , Y_2\rangle \notag \\
			& \qquad -\sum_{i\in I} \sign(\varepsilon_i) \mathbbm{1}_{T_i}(\bar{t}) y_i -L\delta_N\sqrt{m}\geq 0 && \forall \bar{t}\in T_N \label{Constr: discretized_purity_indicator}\\
			& \tilde{b}_{\bar{t} + e_j\delta_N}^i -\tilde{b}_{\bar{t}}^i = \Delta_{\bar{t}}^{-,i,j}-\Delta_{\bar{t}}^{+,i,j} && \forall \bar{t}\in T_N, i\in [k], j\in [m] \label{Constr: jump_def}\\
			& \sum_{\substack{l\in M_\delta:\\ \bar{t}=t_0+le_j}}
			 \Delta_{\bar{t}}^{-,i,j}+\Delta_{\bar{t}}^{+,i,j} \leq 2 && \forall i\in [k], j\in [m], t_0\in T_0^j \label{Constr: sum_delta_bound}\\
			& x_{ij}^-\geq \sum_{\substack{l\in M_\delta:\\ \bar{t}=t_0+le_j}} (l+\delta_N) \Delta_{\bar{t}}^{-,i,j}  && \forall i\in [k],j\in [m], t_0\in T_0^j \label{Constr: a-_lower_bound}\\
			& x_{ij}^+\leq M-\sum_{\substack{l\in M_\delta:\\ \bar{t}=t_0+le_j}} (M-l) \Delta_{\bar{t}}^{+,i,j}  && \forall i\in [k], j\in [m], t_0\in T_0^j \label{Constr: a+_upper_bound}\\
			& x_{ij}^+-x_{ij}^- \geq M \sum_{\substack{l\in M_\delta:\\ \bar{t}=t_0+le_j}} \Delta_{\bar{t}}^{+,i,j}\notag\\
			& \quad\ -\hspace{-0.38cm}\sum_{\substack{l\in M_\delta:\\ \bar{t}=t_0+le_j}} \left( (M-l) \Delta_{\bar{t}}^{+,i,j} - (l+\delta_N) \Delta_{\bar{t}}^{-,i,j}\right) && \forall i\in [k], j\in [m], t_0\in T_0^j \label{Constr: adiff_lower_bound}\\
			& 0 \leq x_{ij}^+ -x_{ij}^-\leq \delta_N (\sum_{\substack{l\in M_\delta:\\ \bar{t}=t_0+le_j}} \tilde{b}_{\bar{t}}^i-1) && \forall i\in [k],\forall j\in [m], t_0\in T_0^j \label{Constr: address_empty_set}\\
			& x_i^-,x_i^+\in C, y \in \R_{\geq 0}^I, Y_1 \in \Scal^{m+1}_{\succeq 0}, Y_2 \in \Scal^m_{\succeq 0}\\
			& \Delta_{\bar{t}}^{-,i,j},\Delta_{\bar{t}}^{+,i,j},\tilde{b}_{\bar{t}}^i\in \{0,1\},
		\end{align}
	\end{subequations}
	where $\tilde{b}_{\bar{t}}^i\coloneqq 0$ for every $\bar{t}\notin T_N$. 
\end{theorem}

We would like to point out, that we could also extend this model further. Indeed, instead of fixed $x_i$ in Theorem \ref{Thm: MIP_multidim}, we could additionally include $x_i$ as a bounded decision variable. This is due to the fact that for bounded $x_i$ the arising bilinear term $x_i\tilde{b}_{\bar{t}}^i$ in Constraint \eqref{Constr: discretized_purity_indicator} can be rewritten as a linear term with the help of additional big-M constraints.

\begin{proof}
	We consider a feasible solution $\Delta_{\bar{t}}^{-,i,j},\Delta_{\bar{t}}^{+,i,j},\tilde{b}_{\bar{t}}^i, x_i^-,x_i^+$ for \eqref{Prob: Dual_Purity_Constraint_indicator_discretized} and show that for every $i\in [k], \bar{t}\in T_N$ we have $\tilde{b}_{\bar{t}}^i=\mathbbm{1}_{[x_i^-,x_i^+]}(\bar{t})$. To this end, note that for every $i\in [k]$ there exists indeed an index $\bar{t}$ with $\tilde{b}_{\bar{t}}^i=1$ due to \eqref{Constr: address_empty_set}. Now, given an arbitrary index $\bar{t}$ with $\tilde{b}_{\bar{t}}^i=1$, we first show that $\tilde{b}_{\bar{t}}^i=1$ implies $\mathbbm{1}_{[x_i^-,x_i^+]}(\bar{t})=1$, i.e., $\bar{t}\in [x_i^-,x_i^+]$:
	
	\bigskip
	We first observe, that for every direction $j$, there exists a $t_0\in T_0^j$ and $\kappa_j\in \{0,\delta_N,2\delta_N,\ldots,M\}$ such that $$\bar{t} = t_0+\kappa_j e_j,$$ 
	i.e., we consider the line in direction $j$ passing through $\bar{t}$ and consequently through $t_0$ as well. Then, we define $\kappa_j^{\max}$ as the index of the last element on this line with $\tilde{b}_t^i=1$, i.e., 
	$$\kappa_j^{\max}\coloneqq \max \{ l\in \{0,\delta_N,2\delta_N,\ldots,M\}: \tilde{b}_{t_0+le_j}^i=1\}.$$
	Thus, $\tilde{b}_{t_0+(\kappa_j^{\max}+\delta_N)e_j}^i=0$ and \eqref{Constr: jump_def} implies $\Delta_{t_0+\kappa_j^{\max}e_j}^{-,i,j}=0, \Delta_{t_0+\kappa_j^{\max}e_j}^{+,i,j}=1$. Moreover, \eqref{Constr: a+_upper_bound} implies 
	\begin{equation}\label{Eq: a+_help}
		x_{ij}^+ \leq M-(M-\kappa_j^{\max})=\kappa_j^{\max}=\bar{t}_j + (\kappa_j^{\max}-\kappa_j),
	\end{equation}
	where the latter equality originates from the definition of $\kappa_j$ above.
	Similarly, we define
	$$\kappa_j^{\min}\coloneqq \min \{l\in \{0,\delta_N,2\delta_N,\ldots,M\}: \tilde{b}_{t_0+le_j}^i=1\}.$$
	Thus, $\tilde{b}_{t_0+(\kappa_j^{\min}-\delta_N)e_j}^i=0$ and \eqref{Constr: jump_def} implies $\Delta_{t_0+(\kappa_j^{\min}-\delta_N)e_j}^{-,i,j}=1, \Delta_{t_0+(\kappa_j^{\min}-\delta_N)e_j}^{+,i,j}=0$. Moreover, \eqref{Constr: a-_lower_bound} implies
	\begin{equation}\label{Eq: a-_help}
		x_{ij}^- \geq (\kappa_j^{\min}-\delta_N)+\delta_N=\kappa_j^{\min} = \bar{t}_j + \kappa_j^{\min}-\kappa_j.
	\end{equation}
	However, due to \eqref{Constr: sum_delta_bound} we know that these are the only nonzero entries for $\Delta_{t_0+le_j}^{-,i,j},\Delta_{t_0+le_j}^{+,i,j}$. Thus due to \eqref{Constr: adiff_lower_bound}, we obtain
	$$x_{ij}^+-x_{ij}^- \geq M - (M-\kappa_j^{\max})-\kappa_j^{\min} = \kappa_j^{\max}-\kappa_j^{\min},$$  
	which implies equality in both \eqref{Eq: a+_help} and \eqref{Eq: a-_help} and thus $\bar{t}_j=\kappa_j\in [\kappa_j^{\min},\kappa_j^{\max}]=[x_{ij}^-, x_{ij}^+]$ for every index $\bar{t}\in T_N$ with $\tilde{b}_{\bar{t}}^i=1$.
	\\
	\\
	For the reverse implication, we need to show that $\bar{t}\in [x_i^-,x_i^+]$ implies $\tilde{b}_{\bar{t}}^i=1$. Due to \eqref{Constr: address_empty_set}, we obtain that $[x_i^-,x_i^+]\neq \emptyset$ implies the existence of a $\bar{t}$ with $\tilde{b}_{\bar{t}}^i=1$. In particular, the previous implication shows that $\bar{t}\in [x_i^-,x_i^+]$.	Beginning with this $\bar{t}$, we prove the following claim for an arbitrary direction $j$:
	
	\begin{equation}\label{Eq: claim_main_Thm}
		\tilde{b}_{\bar{t}}^i=1 \text{ implies } \tilde{b}_{\bar{t}+le_j}^i =1 \text{ for every } l: \bar{t}_j+l\in [x_{ij}^-, x_{ij}^+].
	\end{equation}
	Let $\bar{t}=t_0+\kappa_je_j$ with $t_0\in T_0^j$ as above. Then, with the same definitions for $\kappa_j^{\min},\kappa_j^{\max}$, the arguments from the previous implication, that led to equality in \eqref{Eq: a+_help} and \eqref{Eq: a-_help} imply $\kappa_j^{\min}=x_{ij}^-$, $\kappa_j^{\max}=x_{ij}^+$. Moreover, the definition of $\kappa_j^{\min}, \kappa_j^{\max}$ leads to:
	$$1=\tilde{b}_{t_0+\kappa_j^{\min}e_j}^i=\tilde{b}_{t_0+(\kappa_j^{\min}+\delta_N)e_j}^i=\ldots =\tilde{b}_{t_0+\kappa_j^{\max}e_j}^i=1$$
	with $(t_0+\kappa_j^{\min}e_j)_j=x_{ij}^-$ and $(t_0+\kappa_j^{\max}e_j)_j=x_{ij}^+$.
	Hence, our claim \eqref{Eq: claim_main_Thm} follows and as the direction $j$ was chosen arbitrarily, we obtain that $\mathbbm{1}_{[x_i^-,x_i^+]}(\bar{t})=1$ also implies $\tilde{b}_{\bar{t}}^i=1$.
\end{proof}

Theorem \ref{Thm: MIP_multidim} yields a
sufficient criterion for the DRO constraint to be satisfied. This is a
considerable advantage as to our knowledge no practically efficient alternative approach is readily available. Positive semidefinite optimization is algorithmically
tractable, and recent research has been successful in enhancing global solution algorithms when binary variables are present as well. Nevertheless, solving a binary SDP is still more elaborate than solving binary linear optimization models. As a result, \eqref{Prob: Dual_Purity_Constraint_indicator_discretized} may be computationally too involved even for modern solvers for a large cardinality of $T_N$. For one-dimensional domains $T$ as considered in \cite{Dienstbier2023a} this challenge has been addressed as follows: Instead of bounding the slope of $p_Y$ through its Lipschitz constant $L$, more elaborate bounds that strengthen Lemma \ref{Lemma:inner_approx} reduce the number of
necessary sample points for a good approximation of \eqref{Prob: primal_linear_X_i}. Moreover, due to the one-dimensional domain $T$, instead of a binary SDP, we obtain a binary MIP as an approximation of \eqref{Prob: primal_linear_X_i} that can typically be solved much faster in practice. We next show some preliminary computational results for the SDP model presented here.


\section{Computational Results}
\label{sec:computational_results}

In this section, we show some preliminary computational results 
for solving model \eqref{Prob: Dual_Purity_Constraint_indicator_discretized} via available binary SDP solvers. We restrict ourselves to solving an illustrative toy example that is easily comprehensible.  


\begin{example}[Bin creating problem]
	Given an $m$-dimensional random variable $t\in T=[0,M]^m$, where focus on $m=2, M=1$ here. Let it be distributed according to a distribution that is contained in a set of probability distributions $\Pcal$. We suppose further that the best known estimates for the expectation of $t$ is $\mu=(0,0)^\top$ and $\Sigma=\begin{pmatrix}
		2 & 0.5 \\
		0.5 & 1
	\end{pmatrix}$. Then, we ask for a representative box $[x^-,x^+]\in \R^2$ for the ambiguity set $\Pcal$ as follows:
	\begin{subequations}\label{prob:bin_creating}
		\begin{align}
			\min_{x_1^-,x_1^+,x_2^-,x_2^+} & |x_1^+-x_1^-| + |x_2^+-x_2^-| \label{eq:example_objective}\\
			\text{ s.t.}~ & 0.1 \leq \min_{\P\in \Pcal} \P([x_i^-,x_i^+]) && \forall i\in \{1,2\}, \label{constr:DRO_example}\\ 
			& x_i^+ - x_i^- \geq 0 && \forall i\in \{1,2\},\\
			& x_1^-,x_1^+,x_2^-,x_2^+\in \R_{\geq 0},
		\end{align}
	\end{subequations}
\end{example}
where, due to $\varepsilon_\mu=0.1, \varepsilon_\Sigma=1$, we set 
$$\Pcal=\{\P\in \Mcal(T)_{\geq 0}: \P(T)=1, (\E_{\P}(t)-\mu)^\top \Sigma (\E_{\P}(t)-\mu) \leq 0.1, \text{Var}(t) \preceq\Sigma \}.$$ Since our primary interest lies in the characteristic behavior of the approach, we refrain from introducing additional constraints into the model.

We note that \eqref{eq:example_objective} can be linearized by adding auxiliary variables $z_1,z_2\in \R_{\geq 0}$, the additional constraints $z_i \geq \pm (x_i^+ - x_i^-)$ and by replacing the objective by $z_1 + z_2$. Moreover, to specify the continuous counterpart of the indicator functions in a manner to satisfy the key property \eqref{Eq: Urysohn_approx}, we set
$$\mathbbm{1}_{[x_i^-,x_i^+]}^c(s) = \frac{d_H(\{s\},(-\infty,x_i^--\delta_N] \cup [x_i^+ + \delta_N,\infty))}{d_H(\{s\},[x_i^-,x_i^+]) + d_H(\{s\},(-\infty,x_i^--\delta_N] \cup [x_i^+ + \delta_N,\infty))},$$
where $d_H$ denotes the \emph{Hausdorff distance} on $\R$, i.e., 
$$d_H(X,Y)\coloneqq \max\left\{ \sup_{x\in X}\inf_{y\in Y} |x-y|,\sup_{y\in Y}\inf_{x\in X} |x-y|  \right\}.$$
We note that $\mathbbm{1}_{[x_i^-,x_i^+]}^c(s) \geq \mathbbm{1}_{[x_i^-,x_i^+]}(s)$ since $\mathbbm{1}_{[x_i^-,x_i^+]}^c(s)\geq 0$ and for $s\in [x_i^-,x_i^+]$, we have that $\mathbbm{1}_{[x_i^-,x_i^+]}^c(s)=1$. The continuity is given by the Lemma of Urysohn, see e.g. \cite{munkrestopology}.

As the univariate functions in the maximum term are both nonnegative, the above inequality holds for their respective product as well and we obtain
\begin{align*}
	\mathbbm{1}_{[x^-,x^+]}^c(t) & \coloneqq \prod_{i=1}^m \mathbbm{1}_{[x_i^-,x_i^+]}^c(t_i) \geq \prod_{i=1}^m \mathbbm{1}_{[x_i^-,x_i^+]}(t_i) = \mathbbm{1}_{[x^-,x^+]}(t).
\end{align*}
The indicator function $\mathbbm{1}_{T_i}$ is the constant function having value one and thus continuous for $T_i=T$. Therefore, \eqref{Eq: Urysohn_approx} is satisfied.

Hence, the only parameter not yet determined is the Lipschitz constant $L$ of the polynomial $p_Y(t)=\left\langle \begin{bmatrix}
	\Sigma & t-\EW{}\\
	(t-\EW{})^\top & \varepsilon_{\EW{}}
\end{bmatrix} , Y_1 \right\rangle +\langle (t-\EW{})(t-\EW{})^\top , Y_2\rangle.$
However, if we follow the proof in Lemma \ref{Lemma: Lipschitz_continuity_of_p_indicator}, we obtain an upper bound of
\begin{equation*}
	\lambda_{\min}\left(\begin{bmatrix}
		\Sigma & 0 \\ 0 & \varepsilon_{\EW{}}
	\end{bmatrix}\right) \Tr(Y_1) \leq \sum_{i=1}^k x_i \mathbbm{1}_{X_i}^c(\EW{}) - \sum_{i\in I} \sign(\varepsilon_i) y_i \leq k + y_1 - y_2 \leq 1 + b =1.1
\end{equation*}
due to \eqref{Eq: help3}. Hence, with $\lambda_{\min}\left(\begin{bmatrix} 2 & 0.5 & 0\\ 0.5 & 1 & 0\\ 0 & 0 & 1\end{bmatrix}\right) = \frac{3-\sqrt{2}}{2} \approx 0.8$, we obtain $\Tr(Y_1) \leq 1.4$. Similarly, we obtain an upper bound of 
\begin{equation*}
	\varepsilon_\Sigma \lambda_{\min}(\Sigma) \Tr(Y_2) \leq \sum_{i\in I} \varepsilon_i y_i = -y_1 + y_2 \leq 1
\end{equation*}
due to \eqref{Eq: LP_help}. We have also invested that $\varepsilon = \tau = (-1,1)^\top$. Hence, with $\lambda_{\min}(\Sigma) \approx 0.8$, we obtain $\Tr(Y_2) \leq 1.3$.

With these bounds it is now possible to determine a Lipschitz constant for $p_Y$. Please note, that the following proposition holds for general values of $m$ and can be combined with the above bounds 
$$\Tr(Y_1) \leq \frac{1+b}{	\lambda_{\min}\left(\begin{bmatrix}
		\Sigma & 0 \\ 0 & \varepsilon_{\EW{}}
	\end{bmatrix}\right)} \text{ and } \Tr(Y_2) \leq \frac{1}{\varepsilon_\Sigma \lambda_{\min}(\Sigma)},$$
if $\Sigma \succ 0, \varepsilon_{\EW{}},\varepsilon_\Sigma >0$ in order to determine a potential Lipschitz constant for \eqref{Constr: discretized_purity_indicator}.
\begin{prop}
	Let $t\in T=[0,M]^m$, $\mu_{\min}\coloneqq \min_{i\in [m]}\{\mu_i\}$ and assume that $\mu_{\min} \leq \frac{M}{2}$. Then, the polynomial is Lipschitz continuous on the compact set $T$ with Lipschitz constant
	$$L=2\Tr(Y_1) + (M-\mu_{\min})\Tr(Y_2)2\sqrt{m}.$$
\end{prop}

\begin{proof}
	We first observe that 
	$$p_Y(t) - p_Y(t') = \left\langle \begin{bmatrix}
		0 & t-t'\\
		(t-t')^\top & 0
	\end{bmatrix} , Y_1 \right\rangle +\langle (t-\EW{})(t-\EW{})^\top - (t'-\EW{})(t'-\EW{})^\top , Y_2\rangle.$$
	We continue by bounding the terms separately. For the first term we obtain:
	\begin{align*}
		\left\langle \begin{bmatrix}
			0 & t-t'\\
			(t-t')^\top & 0
		\end{bmatrix} , Y_1 \right\rangle & \leq \|\begin{bmatrix}
		0 & t-t'\\
		(t-t')^\top & 0
		\end{bmatrix}\|_F \|Y_1\|_F\\
		& = \sqrt{2\|t-t'\|_2^2} \sqrt{\Tr(Y_1^\top Y_1)}\\
		& = \sqrt{2} \|t-t'\|_2 \sqrt{\sum_{l=1}^{m+1} \lambda_l^2}\\
		& \leq \sqrt{2} \|t-t'\|_2 \Tr(Y_1),
	\end{align*}
	where the first inequality is due to Cauchy-Schwartz' inequality and the last step is due to the equivalence of norms. Similarly,
	\begin{align*}
		&\langle (t-\EW{})(t-\EW{})^\top - (t'-\EW{})(t'-\EW{})^\top , Y_2\rangle  \\&\leq \sum_{i,j\in [m]} [(t_i-\mu_i)(t_j-\mu_j) - (t_i'-\mu_i)(t_j'-\mu_j)]\\
		& = \Tr(Y_2) \sqrt{\sum_{i,j\in [m]} [t_it_j - t_i\mu_j -\mu_it_j - t_i't_j' + t_i'\mu_j + \mu_it_j']^2}\\
		& = \Tr(Y_2) \sqrt{\sum_{i,j\in [m]} [t_it_j - t_i't_j'  - (t_i-t_i')\mu_j -\mu_i (t_j-t_j')]^2}\\
		& = \Tr(Y_2) \sqrt{\sum_{i,j\in [m]} [(t_i-t_i')t_j + t_i'(t_j-t_j')  - (t_i-t_i')\mu_j -\mu_i (t_j-t_j')]^2}\\
		& = \Tr(Y_2) \sqrt{\sum_{i,j\in [m]}  [(t_j-\mu_j) (t_i-t_i') + (t_i'-\mu_i)(t_j-t_j')]^2}\\
		& \leq \Tr(Y_2)(M-\mu_{\min} ) \sqrt{ \sum_{i,j\in [m]} [(t_i-t_i') + (t_j-t_j')]^2}\\
		& = \Tr(Y_2)(M-\mu_{\min} ) 2\sqrt{m} \|t-t'\|_2,
	\end{align*}
where we utilize Cauchy-Schwartz for the first inequality as above and apply the assumption $\mu_{\min} \leq \frac{M}{2}$. 
\end{proof}
We note that for random vectors where the assumption $\mu_{\min} \leq \frac{M}{2}$ is violated, one may also be able to represent Problem \eqref{prob:bin_creating} appropriately by choosing the domain $[0,M]^m +\mu_{\min}$.

With these notations, the model \eqref{Prob: Dual_Purity_Constraint_indicator_discretized} could now be solved by an appropriate binary SDP solver. For this specific toy example, however, the problem structure can be exploited further to lead to a smaller model.
Namely, we simplify the safe approximation \eqref{Prob: Dual_Purity_Constraint_indicator_discretized} to the following model
\begin{subequations}\label{prob:safe_approx_for_bin_creating}
	\begin{align}
		\min\ & x^+_1-x^-_1 + x^+_2-x^-_2\\
		\st~\ & -y_1+y_2 -\varepsilon_\Sigma \langle \Sigma, Y_2 \rangle \geq b \\
		& \tilde{b}_{\bar{t}} - \left\langle \begin{bmatrix}
			\Sigma & \bar{t}-\EW{}\\
			(\bar{t}-\EW{})^\top & \varepsilon_1
		\end{bmatrix} , Y_1 \right\rangle \notag\\
		& \qquad +\langle (\bar{t}-\EW{})(\bar{t}-\EW{})^\top , Y_2\rangle \notag \\
		& \qquad -\sum_{i\in I} \sign(\varepsilon_i) \mathbbm{1}_{T_i}(\bar{t}) y_i -L\delta_N\sqrt{m}\geq 0 && \forall \bar{t}\in T_N \label{constr:dual_purity_bin_creating}\\
		& x_{j}^+ - x_{j}^- \geq \delta_N (\sum_{\substack{l\in M_\delta:\\ \bar{t}=t_0+le_j}} \tilde{b}_{\bar{t}}^i-1) && \forall i\in [k],\forall j\in [m], t_0\in T_0^j \\
		& x^-,x^+\in [0,M]^m, y \in \R_{\geq 0}^I, Y_1 \in \Scal^{m+1}_{\succeq 0}, Y_2 \in \Scal^m_{\succeq 0}\\
		& \tilde{b}\in \{0,1\}^{T_N}.
	\end{align}
\end{subequations}
The problem formulation \eqref{prob:safe_approx_for_bin_creating} strongly depends on the discretization width. On the one hand, if $\delta_N$ is too large, the safety term $L\delta_N\sqrt{m}$ in \eqref{constr:dual_purity_bin_creating} leads to an infeasible model. On the other hand, if $\delta_N$ is too small, the number of binary variables $\tilde{b}_{\bar{t}}$ in \eqref{prob:safe_approx_for_bin_creating} grows polynomially in this width, for $m=2$ quadratically.
Despite the considerable progress in solving binary SDPs in recent years, these models are still demanding for a large number of binary variables. We thus vary the discretization width in our experiments.  


The computational experiments have been performed with the parameters $c,\varepsilon_i,\varepsilon_\mu, \varepsilon_\Sigma, \mu, \Sigma, L$ and $M$ chosen as above. In order to solve the resulting instance of ~\eqref{prob:safe_approx_for_bin_creating}, we utilized the state-of-the-art solver SCIP-SDP as presented in
e.g. \cite{Gally2018a}, where the SDP subproblems have been solved by MOSEK \cite{mosek}. The computations have been executed on a MacBookAir 2024 with M3 chip and 16GB memory.


\begin{figure}[h]
	\pgfplotsset{width=6.5cm,compat=1.18}
	\pgfplotsset{domain=0:0.1}
	\begin{tikzpicture}[baseline]
		\begin{axis}[xlabel = objective value with decreasing $\delta_N$, x dir=reverse, xticklabel style={/pgf/number format/fixed}]
			\addplot+[sharp plot,] coordinates {
				(0.1,2) (0.083,2) (0.066666, 1.8) (0.05, 1.7) (0.04, 1.64)
			};
		\end{axis}
	\end{tikzpicture}
	\begin{tikzpicture}[baseline]
		\begin{axis}[xlabel = runtime with decreasing $\delta_N$ in seconds, x dir=reverse, xticklabel style={/pgf/number format/fixed}]
			\addplot+[sharp plot,] coordinates {
				(0.1,1) (0.083,1) (0.066666, 12) (0.05, 12) (0.04, 400)
			};
		\end{axis}
	\end{tikzpicture}
	\caption{Development of objective values and runtime for $M=1$, $b=0.1$, $\varepsilon_\mu=0.1$, $\varepsilon_\Sigma=1$, $\mu=0$ and $\Sigma=\begin{bmatrix}
			2 & 0.5\\ 0.5 & 1
	\end{bmatrix}$.}
\label{fig:compres}
\end{figure}

Figure~\ref{fig:compres} displays the objective function value of the objective as a function of discretization width $\delta_N$ (left) and the running time as a function of $\delta_N$ (right). Here, the choice of $\delta_N$ is crucial as the number of binary variables in \eqref{prob:safe_approx_for_bin_creating} is given by $\left(\frac{M}{\delta_N}\right)^m$, here $\left(\frac{1}{\delta_N}\right)^2$. 

It is worth noting that the solution $x^-_i=-1, x^+_i=1, \tilde{b}=1, y_2=b$ and vanishing $Y_1,Y_2,y_i$ for $i\neq 2$ is feasible if $1-b-L\delta_N \sqrt{m} \geq 0$ $\delta_N \leq \frac{1-b}{L\sqrt{m}}$. Thus, choosing $\delta_N \leq \frac{1-b}{L\sqrt{m}}$ guarantees a feasible safe approximation, but smaller values for $\delta_N$ reduce the safety term $L\delta_N\sqrt{m}$ and thereby enable a less conservative approximation of the original problem \eqref{prob:bin_creating}. From the left figure, we observe that this allows for a smaller box to capture the required probability mass of $10\%$.

However, as the number of binary variables increases at a rate of $\left(\frac{1}{\delta_N}\right)^2$, the required running time quickly increases as well, as can be seen in the right figure. While model~\eqref{prob:safe_approx_for_bin_creating} can be solved to global optimality within a few seconds for a discretization width of up to $0.05$, the required running time for $\delta_N=0.04$ is about 400 seconds. As the reduction in objective function value slows down for smaller discretization width, it can be assumed that a discretization of $\delta_N=0.04$ reasonably balances running time and quality of the safe approximation. We do not display smaller discretization widths here as the corresponding runtime exceeded a limit of $1$h.

This section has served to illustrate the safe approximation via an academic example. It is evident that the model \eqref{Prob: Dual_Purity_Constraint_indicator_discretized} as well as its variant~\eqref{prob:safe_approx_for_bin_creating} can be computationally demanding, in particular for small discretization width. There are two future research directions to mitigate this: First, binary SDP solvers are expected to improve further over time, as this is an active area of research. Second, model~\eqref{prob:safe_approx_for_bin_creating} is a generic formulation designed to accommodate the general problem structure. To the best of our knowledge, it is the first such safe approximation for multivariate simple functions presented in the literature. By exploiting specific problem structures, it is expected that the formulation can be made smaller, thus leading to an improved solvability.




\section{Conclusion}
\label{Sec: Conclusion}

In this paper, we present an extension of the novel approach in \cite{Dienstbier2023a} for distributionally robust optimization problems to cases, where multivariate simple functions are allowed. As simple functions can be included in the model, the presented approximation pushes the applicability of duality-based reformulations of distributional robustness significantly beyond convexity. Moreover, early convergence results from \cite{Dienstbier2023a} for univariate indicator functions indicate, that the presented approximation may converge to the actual optimum. A proof for this convergence as well as an extension from simple functions to more general functions is a desirable goal for future research.

With respect to algorithmic tractablilty, we have shown that a suitably discretized safe
approximation yields a mixed-integer positive-semidefinite optimization
model making it eligible for recent MISDP approaches as presented in
e.g. \cite{Gally2018a} or the YALMIP framework
\cite{Lofberg2004a}. Thus, the presented formulations is tractable by
using state-of-the-art solvers for MISDP, which we have also shown by computational results for an academic example. As running times strongly scale with the discretization width, future research will aim to downsizing the model without incurring loss in obtained quality.

\section*{Acknowledgments}
\label{sec:acknowledgements}
%
The paper is funded by the Deutsche Forschungsgemeinschaft (DFG, German Research Foundation) - Project-ID 416229255 - SFB 1411.

\printbibliography

\end{document}